\documentclass{amsproc}

\newtheorem{thm}{Theorem}[section]

\newtheorem{prop}[thm]{Proposition}

\newtheorem*{tha}{Theorem A}
\newtheorem*{thb}{Theorem B}
\newtheorem*{thc}{Theorem C}
\theoremstyle{remark}

\newcommand{\C}{{\mathbb C}}
\newcommand{\D}{{\mathbb D}}
\newcommand{\R}{{\mathbb R}}
\newcommand{\T}{{\mathbb T}}

\renewcommand{\P}{{\mathcal P}}

\newcommand{\La}{\Lambda}
\newcommand{\la}{\lambda}

\newcommand{\f}{\frac}
\newcommand{\ov}{\overline}
\newcommand{\const}{\text{\rm const}}
\newcommand{\eps}{\varepsilon}

\newcommand{\al}{\alpha}

\newcommand{\ga}{\gamma}
\newcommand{\de}{\delta}
\renewcommand{\th}{\theta}
\newcommand{\ph}{\varphi}

\newcommand{\ze}{\zeta}

\newcommand{\om}{\omega}

\renewcommand{\sb}{\subset}

\numberwithin{equation}{section}

\begin{document}

\title{$ABC$-type estimates via Garsia-type norms}

\author{Konstantin M. Dyakonov}
\address{ICREA and Universitat de Barcelona, 
Departament de Matem\` atica Aplicada i An\` alisi, 
Gran Via 585, E-08007 Barcelona, Spain} 
\email{konstantin.dyakonov@icrea.es}
\keywords{Mason's theorem, $abc$ conjecture, Garsia-type norm, Lipschitz spaces, Blaschke products} 
\subjclass[2000]{30D50, 30D55, 11D41.} 
\thanks{Supported in part by grant MTM2008-05561-C02-01 from El Ministerio de Ciencia 
e Innovaci\'on (Spain) and grant 2009-SGR-1303 from AGAUR (Generalitat de Catalunya).}

\begin{abstract} We are concerned with extensions of the Mason--Stothers $abc$ theorem from 
polynomials to analytic functions on the unit disk $\D$. The new feature is that the number 
of zeros of a function $f$ in $\D$ gets replaced by the norm of the associated Blaschke 
product $B_f$ in a suitable smoothness space $X$. Such extensions are shown to exist, and the 
appropriate $abc$-type estimates are exhibited, provided that $X$ admits a \lq\lq Garsia-type norm", 
i.\,e., a norm sharing certain properties with the classical Garsia norm on $\text{\rm BMO}$. Special 
emphasis is placed on analytic Lipschitz spaces. 
\end{abstract}

\maketitle

\section{Introduction}

One of the famous (and notoriously difficult) open problems in number theory is the so-called 
{\it $abc$ conjecture} of Masser and Oesterl\`e. It states that to every $\eps>0$ there is 
a constant $K(\eps)$ with the following property: whenever $a$, $b$ and $c$ are relatively prime 
positive integers satisfying $a+b=c$, one has 
$$c\le K(\eps)\cdot\{\text{\rm rad}(abc)\}^{1+\eps}.$$ 
Here, $\text{\rm rad}(\cdot)$ stands for the {\it radical} of the integer in question, defined as 
the product of the distinct primes that divide it. (In other words, $\text{\rm rad}(m)$ 
is the greatest square-free divisor of $m$.) We refer to \cite{GT, L} for a discussion 
of the $abc$ conjecture and its potential applications. 
\par The conjecture was inspired by the following {\it $abc$ theorem} for polynomials. When 
stating it, we write $\deg{p}$ for the degree of a polynomial $p$ (in one complex variable) 
and $\widetilde N(p)$ for the number of its distinct zeros in $\C$. 

\begin{tha} Suppose $a$, $b$ and $c$ are polynomials, not all constants, having no common 
zeros and satisfying $a+b=c$. Then 
\begin{equation}\label{eqn:masonnn}
\max\{\deg{a},\,\deg{b},\,\deg{c}\}\le\widetilde N(abc)-1.
\end{equation}
\end{tha}

\par This result is also known as Mason's theorem. It is indeed contained -- in a more general form -- in 
Mason's book \cite{M}, but the current version is essentially due to Stothers \cite{St}. The situation seems 
to be in full accordance with {\it V. I. Arnold's principle}: a personal name, when attached to a mathematical 
notion or statement, is never the name of the true discoverer. (Needless to say, Arnold's principle applies 
to itself as well.) 

\par Various approaches to Theorem A can be found in \cite{GT, GH, L, ShSm}. Let us also mention the 
following generalization involving any finite number of polynomials; see \cite{BM, DLoc, GH} for this 
and other related results. 

\begin{thb} Let $p_0,\dots,p_n$ be linearly independent polynomials and put $p_{n+1}=p_0+\dots+p_n$. 
Assume further that the zero-sets of $p_0,\dots,p_{n+1}$ are pairwise disjoint. Then 
\begin{equation}\label{eqn:abcxyz}
\max\{\deg{p_0},\dots,\deg{p_{n+1}}\}\le n\widetilde N(p_0p_1\dots p_{n+1})-\f{n(n+1)}2.
\end{equation}
\end{thb}

\par Quite recently, in \cite{DCRM, DLoc}, we came up with some $abc$-type estimates that make sense in a much 
more general setting. Namely, we were concerned with analytic functions on a (reasonably decent) planar domain, 
rather than just polynomials on $\C$. In fact, \cite{DCRM} dealt with the case of the disk only, while 
the functions were assumed to be analytic in a neighborhood of its closure. Retaining these hypotheses, 
we now go on to describe part of what we did in \cite{DCRM}. 

\par Write $\D$ for the unit disk $\{z\in\C:|z|<1\}$ and $\T$ for its boundary, $\partial\D$. Suppose 
$f_0,\dots,f_n$ are functions that are analytic on the (closed) disk $\D\cup\T$, and set 
\begin{equation}\label{eqn:summm}
f_{n+1}=f_0+\dots+f_n. 
\end{equation}
For each $j=0,\dots,n+1$, we associate with $f_j$ the (finite) Blaschke product $B_j$ built from 
the function's zeros. This means that $B_j$ is given by 
\begin{equation}\label{eqn:blaschkeee}
z\mapsto\prod_{k=1}^s\left(\f{z-a_k}{1-{\ov a_k}z}\right)^{m_k},
\end{equation}
where $a_k=a_k^{(j)}$ ($1\le k\le s=s_j$) are the distinct zeros of $f_j$ in $\D$, and $m_k=m_k^{(j)}$ 
are their respective multiplicities. Further, we let $\mathbf B$ denote the {\it least common multiple} 
(defined in the natural way) of the Blaschke products $B_0,\dots,B_{n+1}$, to be written as 
$$\mathbf B:=\text{\rm{LCM}}(B_0,\dots,B_{n+1}),$$
and we put 
$$\mathcal B:=\text{\rm{rad}}(B_0B_1\dots B_{n+1}).$$
In the latter formula, we use the notation $\text{\rm{rad}}(B)$ for the {\it radical} 
of a Blaschke product $B$. This is, by definition, the Blaschke product 
that arises when the zeros of $B$ are all converted into simple ones. 
In other words, given a Blaschke product of the form \eqref{eqn:blaschkeee}, its radical 
is obtained by replacing each $m_k$ with $1$. 
\par Finally, let $W=W(f_0,\dots,f_n)$ be the {\it Wronskian} of the (analytic) 
functions $f_0,\dots,f_n$, so that 
\begin{equation}\label{eqn:wronskiannn}
W:=
\begin{vmatrix}
f_0&f_1&\dots&f_n\\
f'_0&f'_1&\dots&f'_n\\
\dots&\dots&\dots&\dots\\
f_0^{(n)}&f_1^{(n)}&\dots&f_n^{(n)}
\end{vmatrix}.
\end{equation}
We then introduce the quantities 
$$\kappa=\kappa(W):=\|W'\|_1\|1/W\|_\infty$$ 
and 
$$\mu=\mu(W):=\|W\|_\infty\|1/W\|_\infty.$$ 
(Here and below, $\|\cdot\|_p$ stands for $\|\cdot\|_{L^p(\T)}$, the $L^p$-norm 
with respect to the {\it normalized} arclength measure on the circle.) 
The two quantities are finite, provided that $W$ has no zeros on $\T$. 
\par Now we are in a position to state the following result from \cite{DCRM}. 

\begin{thc} Suppose $f_j$ ($j=0,1,\dots,n+1$) are analytic functions 
on $\D\cup\T$, related by \eqref{eqn:summm} and such that 
the Wronskian \eqref{eqn:wronskiannn} vanishes nowhere on $\T$. 
Once $\mathbf B$, $\mathcal B$, $\kappa$ and $\mu$ are defined as above, we have 
\begin{equation}\label{eqn:harrr}
N_\D(\mathbf B)\le\kappa+n\mu N_\D(\mathcal B),
\end{equation}
where $N_\D(\cdot)$ denotes the number of the function's zeros in $\D$, counting multiplicities. 
\end{thc}

\par It was explained in \cite{DCRM} how to derive the original $abc$ inequality \eqref{eqn:masonnn} 
from \eqref{eqn:harrr}. Basically, the idea is to apply Theorem C, with $n=1$, to the three polynomials, 
rescaling everything for the disk $R\D=\{z:|z|<R\}$, and then pass to the limit as $R\to\infty$. 
In the case of an arbitrary $n$, we similarly deduce Theorem B from Theorem C; see \cite{DLoc} for 
details. 

\par Also, in \cite{DCRM}, inequality \eqref{eqn:harrr} was supplemented with a certain alternative 
estimate, which we do not cite here. Further developments, as contained in \cite{DLoc}, included the 
situation where the functions $f_j$ are merely analytic on $\D$ and suitably smooth up to $\T$ (but 
not necessarily analytic on $\D\cup\T$); in particular, the case of infinitely many zeros was dealt with. 
In addition, other -- fairly general -- domains were considered in place of the disk. 

\par In this note, we extend \eqref{eqn:harrr} in yet another direction. Let us observe that, for 
a Blaschke product $B$, the number of its zeros $N_\D(B)$ coincides with the quantity $\|B'\|_1$; 
moreover, a similar quantity $\|W'\|_1$ appears in the definition of the coefficient $\kappa$ 
above.  
Therefore, \eqref{eqn:harrr} reflects a certain fact about the Hardy--Sobolev space 
$H^1_1:=\{f\in H^1:\,f'\in H^1\}$, equipped with the norm $\|f\|_{H^1_1}:=\|f'\|_1$. 
We may ask, then, what other \lq\lq smooth" analytic spaces $X$ admit 
(under the hypotheses of Theorem C) the $abc$-type estimate 
\begin{equation}\label{eqn:abcxxx}
c\|\mathbf B\|_X\le\kappa_X+n\mu\|\mathcal B\|_X,
\end{equation}
with $\kappa_X:=\|W\|_X\|1/W\|_\infty$ and with some constant $c=c_X>0$. While for $X=H^1_1$ 
we have $c=1$, it seems reasonable to allow for an unspecified factor $c$ in the general case; 
this $c$ should depend neither on $n$ nor on the functions involved. 

\par Our main result, to be stated in Section 3, will provide us with a collection of spaces 
$X$ that enjoy the required property. Roughly speaking, the \lq\lq smoothest" of these is the 
space $H^\infty_1:=\{f\in H^\infty:\,f'\in H^\infty\}$ (i.e., the class of analytic functions 
satisfying the Lipschitz 1-condition) with the natural norm $\|f\|_{H^\infty_1}:=\|f'\|_\infty$. 
Indeed, while $H^\infty_1$ does admit the $abc$-type estimate \eqref{eqn:abcxxx}, 
no further increase of smoothness (e.g., in the sense of passing to higher order Lipschitz 
spaces) is possible. We first state the positive part of that endpoint result as follows. 

\begin{prop}\label{thm:prop1} Under the hypotheses of Theorem C, one has 
\begin{equation}\label{eqn:lipone}
c\|\mathbf B'\|_\infty\le\|1/W\|_\infty\cdot\left(\|W'\|_\infty
+n\|W\|_\infty\|\mathcal B'\|_\infty\right),
\end{equation}
where $c>0$ is an absolute constant. 
\end{prop}

\par This means that the $abc$-type inequality 
\begin{equation}\label{eqn:genver}
c\|\mathbf B\|_X\le\|1/W\|_\infty\cdot\left(\|W\|_X+n\|W\|_\infty\|\mathcal B\|_X\right), 
\end{equation}
or equivalently \eqref{eqn:abcxxx}, holds with $X=H^\infty_1$. 
Moreover, we shall see that there is a certain \lq\lq privileged" norm on $H^\infty_1$ which, 
when used in place of $\|f'\|_\infty$ above, makes the corresponding statement true with $c=1$. 

\par It is precisely the implementation of a special norm that is crucial to our approach. 
Our method also applies to the analytic Lipschitz spaces $A_\om$ associated 
with certain slower majorants $\om$, not just to $H^\infty_1$ (in which 
case the majorant is $\om_1(t):=t$). Here, by saying that $\om$ is a {\it majorant} 
we mean that $\om:\R_+\to\R_+$ is an increasing continuous 
function on $\R_+:=(0,\infty)$ with $\lim_{t\to0^+}\om(t)=0$ such that 
$\om(t)/t$ is nonincreasing for $t>0$. The space $\La_\om(E)$ on a set $E\subset\C$ is then 
formed by the functions $f:E\to\C$ satisfying 
$$\|f\|_{\La_\om(E)}:=\sup\left\{\f{|f(z_1)-f(z_2)|}{\om(|z_1-z_2|)}:\,z_1,z_2\in E,
\,z_1\ne z_2\right\}<\infty.$$ 
When $\om$ is of the form $\om_\al(t):=t^\al$, with $0<\al\le1$, we write $\La^\al(E)$ 
rather than $\La_{\om_\al}(E)$. Further, we define 
the analytic Lipschitz space $A_\om$ to be $H^\infty\cap\La_\om(\D)$ 
and endow it with the norm $\|\cdot\|_{\La_\om}:=\|\cdot\|_{\La_\om(\D)}$. When 
$\om=\om_\al$, the corresponding $A_\om$-space is denoted by $A^\al$; thus, in particular, 
$H^\infty_1=A^1$ (with equality of norms). 

\par Finally, we recall that a majorant $\om$ is said to be {\it regular} if 
$$\int_0^\de\f{\om(t)}tdt+\de\int_\de^\infty\f{\om(t)}{t^2}dt\le C\om(\de),\qquad0<\de<2,$$ 
for some fixed $C=C_\om>0$. The basic examples of regular majorants are the $\om_\al$'s 
with $0<\al<1$. 

\begin{prop}\label{thm:prop2} For every regular majorant $\om$ one has, 
under the hypotheses of Theorem C, 
\begin{equation}\label{eqn:lipreg}
c_\om\|\mathbf B\|_{\La_\om}\le\|1/W\|_\infty\cdot\left(\|W\|_{\La_\om}
+n\|W\|_\infty\|\mathcal B\|_{\La_\om}\right),
\end{equation}
where $c_\om>0$ is a constant depending only on $\om$. 
\end{prop}

\par Both propositions will follow as special cases from our main result, Theorem \ref{thm:mainresulttt} 
in Section 3 below. Indeed, the hypothesis of that theorem (concerning the existence of a \lq\lq Garsia-type 
norm" on the space in question) is fulfilled for our Lipschitz spaces $A_\om$ and $H^\infty_1$; this is 
explained in Section 2. 

\par We conclude this introduction with an example showing that the $H^\infty_1$ result 
(Proposition \ref{thm:prop1}) is both sharp and best possible, at least within the Lipschitz scale, 
as we said before. Consider the functions $f_0(z)=1$ 
and $f_j(z)=\eps z^j/j!$ for $j=1,\dots,n$, with a suitable $\eps>0$; then 
define $f_{n+1}$ by \eqref{eqn:summm}. If $\eps$ is small enough, then $f_{n+1}$ is zero-free on 
$\D$. The Blaschke products that arise are $B_0(z)=B_{n+1}(z)=1$ and $B_j(z)=z^j$ for $1\le j\le n$. 
We have, therefore, $\mathbf B(z)=z^n$ and $\mathcal B(z)=z$, whence $\|\mathbf B'\|_\infty=n$ and 
$\|\mathcal B'\|_\infty=1$. Also, the Wronskian matrix being upper triangular, one easily 
finds that $W=\eps^n(=\const)$; this yields $\|W'\|_\infty=0$ and $\|W\|_\infty\|1/W\|_\infty=1$. 
Consequently, equality holds in \eqref{eqn:lipone} with $c=1$. 

\par Now let $1<\al<2$ and consider the (higher order) Lipschitz space $A^\al:=\{f\in H^\infty:\,
f'\in A^{\al-1}\}$, normed in a natural way. This time, our current $f_j$'s 
provide a counterexample to the $abc$-type estimate \eqref{eqn:genver} 
with $X=A^\al$. Indeed, the left-hand side in \eqref{eqn:genver} is now a constant 
times $\|z^n\|_{A^\al}$, which is comparable to $n^\al$. As to the right-hand side, it is still 
$O(n)$, so the inequality breaks down. 

\section{Preliminaries on Garsia-type norms} 

The classical {\it Garsia norm} on the space $\text{\rm BMOA}:=\text{\rm BMO}\cap H^2$ is 
given by 
\begin{equation}\label{eqn:garsia}
\|f\|_G:=\sup_{z\in\D}\left\{\mathcal P(|f|^2)(z)-|f(z)|^2\right\}^{1/2}.
\end{equation} 
Here, the notation $\mathcal P\ph$ stands for the Poisson integral of a function $\ph\in L^1(\T)$, so that 
$$\mathcal P\ph(z):=\f1{2\pi}\int_\T \ph(\ze)\f{1-|z|^2}{|\ze-z|^2}|d\ze|,\qquad z\in\D.$$ 
The definition \eqref{eqn:garsia} actually makes sense for all $f\in H^2$, and it is well known that 
$\text{\rm BMOA}=\{f\in H^2:\|f\|_G<\infty\}$. Moreover, the Garsia norm $\|\cdot\|_G$ is equivalent 
to the original $\text{\rm BMO}$-norm $\|\cdot\|_*$ defined in terms of mean oscillation; see 
\cite[Chapter VI]{G} or \cite[Chapter X]{K}. 
\par Recently, we introduced in \cite{DAdv} a more general concept of {\it Garsia-type norm} 
(GTN) as follows. Suppose $X$ is a Banach space of analytic functions on $\D$ 
such that $X\sb H^p$ for some $p>0$. Write $|H^p|$ for the set of all nonnegative 
functions $g\in L^p(\T)$ satisfying either $\log g\in L^1(\T)$ or $g=0$ a.e.; these are 
precisely (the boundary values of) the moduli of $H^p$-functions. Further, assume that 
there exist a mapping $\Psi:|H^p|\times\D\to[0,+\infty]$ and a function 
$k:\D\to\R_+$ with the following properties: 

\smallskip$\bullet$ $\Psi(\la g,z)=\la^p\Psi(g,z)$ whenever $\la\in\R_+$, $g\in|H^p|$ 
and $z\in\D$, 

\smallskip$\bullet$ $\Psi(|f|,z)\geq|f(z)|^p$ for all $f\in H^p$ and $z\in\D$, 

\smallskip$\bullet$ the quantity 
$$\mathcal N(f)=\mathcal N_{p,\Psi,k}(f):=\sup_{z\in\D}\f{\{\Psi(|f|,z)
-|f(z)|^p\}^{1/p}}{k(z)},\qquad f\in H^p,$$ 
is comparable to $\|f\|_X$ with constants not depending on $f$ (it is 
understood that $\|f\|_X=\infty$ for $f\in H^p\setminus X$). 

\smallskip Then we say that $\mathcal N(\cdot)$ {\it is a GTN on $X$}, and accordingly, that 
{\it $X$ admits a GTN}. In fact, the $X$'s we have in mind will always be analytic subspaces 
of certain smoothness classes on $\T$, and the constants will have zero $X$-norm. Therefore, 
we should also have $\mathcal N(\mathbf1)=0$ (where $\mathbf1$ is the constant function $1$), 
and this reduces to saying that 
\begin{equation}\label{eqn:psioneone}
\Psi(\mathbf1,z)=1,\qquad z\in\D, 
\end{equation}
an assumption to be imposed hereafter. 

\par The reason why GTN's are useful is that, once available, such a norm makes it easy 
to separate the contributions of the two factors in the canonical (inner-outer) 
factorization of functions in $X$. Indeed, let $f=h\th$, where 
$h\in H^p$ and $\th$ is an inner function. Since $|f|=|h|$ a.e. 
on $\T$, we have 
$$\Psi(|f|,z)-|f(z)|^p=\left\{\Psi(|h|,z)-|h(z)|^p\right\}+
\{|h(z)|^p(1-|\th(z)|^p)\}$$ 
for $z\in\D$. Both terms in curly brackets on the right are nonnegative, so 
$$\{\Psi(|f|,z)-|f(z)|^p\}^{1/p}\asymp 
\left\{\Psi(|h|,z)-|h(z)|^p\right\}^{1/p}
+|h(z)|(1-|\th(z)|^p)^{1/p}$$ 
(meaning that the ratio of the two sides lies between two positive constants that depend 
only on $p$), and hence 
$$\mathcal N_{p,\Psi,k}(f)\asymp\mathcal N_{p,\Psi,k}(h)+\mathcal S_{p,k}(h,\th),$$
where 
$$\mathcal S_{p,k}(h,\th):=\sup_{z\in\D}\f{|h(z)|(1-|\th(z)|^p)^{1/p}}{k(z)}.$$
More precisely, 
\begin{equation}\label{eqn:asympforbis}
\max\left\{\mathcal N_{p,\Psi,k}(h),\,\mathcal S_{p,k}(h,\th)\right\}
\le\mathcal N_{p,\Psi,k}(f)\le c_p\left\{\mathcal N_{p,\Psi,k}(h)
+\mathcal S_{p,k}(h,\th)\right\}
\end{equation}
for a suitable constant $c_p>0$; we can take $c_p=1$ if $p\ge1$. 
\par We see, in particular, that 
\begin{equation}\label{eqn:divproppp}
\mathcal N_{p,\Psi,k}(h\th)\ge\mathcal N_{p,\Psi,k}(h) 
\end{equation}
for $h\in H^p$ and $\th$ inner, which means that division by inner factors preserves 
membership in $X$. (Equivalently, $X$ enjoys the so-called {\it f-property}.) 
On the other hand, given $h\in X$ and an inner function $\th$, we have 
$h\th\in X$ if and only if $\mathcal S_{p,k}(h,\th)<\infty$. When $h=\mathbf1$, this 
gives a criterion for an inner function $\th$ to be in $X$; moreover, \eqref{eqn:psioneone} 
implies that 
$$\mathcal N_{p,\Psi,k}(\th)=\mathcal S_{p,k}(\mathbf1,\th).$$ 

\par The parameters corresponding to the Garsia norm $\|\cdot\|_G$ are obviously 
$p=2$, $\Psi(g,z)=\P(g^2)(z)$ and $k(z)\equiv1$. Keeping the same $p$ and $\Psi$ while 
putting $k(z)=\om(1-|z|)$, with a majorant $\om$, gives rise to the GTN 
$$\|f\|_{G,\om}:=\sup_{z\in\D}\f{\left\{\mathcal P(|f|^2)(z)-|f(z)|^2\right\}^{1/2}}{\om(1-|z|)}$$ 
on the space $\text{\rm BMOA}_\om:=\{f\in H^2:\|f\|_{G,\om}<\infty\}$. If $\om(t)$ tends 
to $0$ slowly enough as $t\to0^+$ (e.g., if $\om(t)=(\log\f et)^{-\eps}$ with a suitably small $\eps>0$), 
then $\text{\rm BMOA}_\om$ will retain many features of $\text{\rm BMOA}$; in particular, it will 
contain unbounded and discontinuous functions. For faster $\om$'s (such as $\om(t)=t^\al$ with 
$0<\al<\f12$), it becomes a Lipschitz space. In fact, it was proved in \cite{DActa} that 
if $\om$ and $\om^2$ are both regular majorants, then $\text{\rm BMOA}_\om$ coincides with $A_\om$, 
the norm $\|\cdot\|_{G,\om}$ being equivalent to $\|\cdot\|_{\La_\om}$. The assumption on $\om^2$ 
cannot be dropped here: just note that the identity function $f_0(z):=z$ has $\|f_0\|_{G,\om}=\infty$ 
whenever $\om(t)/\sqrt t\to0$ as $t\to0^+$. 

\par Furthermore, assuming that $\om$ alone is a regular majorant, we proved in \cite{DActa} that 
the functional 
$$\mathcal M_\om(f):=\|\,|f|\,\|_{\La_\om(\T)}+\sup_{z\in\D}\f{(\P|f|)(z)
-|f(z)|}{\om(1-|z|)}$$ 
provides an equivalent norm on $A_\om$. In particular, this is the case for $\om(t)=t^\al$ 
with $0<\al<1$ (but not with $\al=1$). Clearly, we have 
$\mathcal M_\om(f)=\mathcal N_{1,\Psi,k}(f)$ with 
$$\Psi(g,z)=\|g\|_{\La_\om(\T)}\cdot\om(1-|z|)+(\P g)(z)$$
and $k(z)=\om(1-|z|)$, so $\mathcal M_\om$ is a GTN on $A_\om$. 

\par Finally, the extreme case $\om(t)=t$ was studied in \cite{DAdv}. There we 
showed that a GTN can be defined on $H^\infty_1$ by taking $p=1$, $k(z)=1-|z|$ and 
$$\Psi(g,z)=\|g\|_{\La^1(\T)}\cdot(1-|z|)+\left|g\left(\f z{|z|}\right)
-\exp\{(\P\log g)(z)\}\right|+\exp\{(\P\log g)(z)\}$$ 
(with the understanding that $\Psi(g,z)=\infty$ if $g\notin\La^1(\T)$, 
and $z/|z|=1$ if $z=0$). The norm that arises is thus 
$\tilde{\mathcal N}_1(\cdot):=\mathcal N_{1,\Psi,k}(\cdot)$, or equivalently, 
$$\tilde{\mathcal N}_1(f)=\|\,|f|\,\|_{\La^1(\T)}
+\sup_{z\in\D}\f{\left|\,|f(z/|z|)|-|\mathcal O_{|f|}(z)|\,\right|
+|\mathcal O_{|f|}(z)|-|f(z)|}{1-|z|},$$ 
where $\mathcal O_{|f|}$ is the outer function with modulus $|f|$ on $\T$. 

\par At the same time, it turns out \cite{DAdv} that the higher order $A^\al$-spaces (i.e., those 
with $\al>1$) or, more generally, the classes $A^n_\om:=\left\{f:f^{(n)}\in A_\om\right\}$ 
with $n\ge1$ admit no GTN at all. This accounts for the special role of the endpoint space $H^\infty_1$ 
in our story. 

\section{Main result} 

\begin{thm}\label{thm:mainresulttt} Let $X$ be a space that admits a Garsia-type norm 
$\mathcal N=\mathcal N_{p,\Psi,k}$, where $p\ge1$ and $\Psi$ satisfies \eqref{eqn:psioneone}. 
Assume that the functions $f_j$ ($j=0,1,\dots,n+1$), related by \eqref{eqn:summm}, are analytic 
on $\D\cup\T$ and that the Wronskian \eqref{eqn:wronskiannn} vanishes nowhere on $\T$. 
Write 
\begin{equation}\label{eqn:lcmradbisss}
\mathbf B:=\text{\rm{LCM}}(B_0,\dots,B_{n+1})\quad\text{and}\quad
\mathcal B:=\text{\rm{rad}}(B_0B_1\dots B_{n+1}), 
\end{equation}
where $B_j$ is the Blaschke product associated with $f_j$. Then 
\begin{equation}\label{eqn:estgar}
\mathcal N(\mathbf B)\le\ga+\mu n^{1/p}\mathcal N(\mathcal B),
\end{equation}
where 
$$\ga=\ga(W):=\mathcal N(W)\|1/W\|_\infty
\quad\text{and}\quad\mu=\mu(W):=\|W\|_\infty\|1/W\|_\infty.$$ 
Also, for some constant $c=c(X,p,\Psi,k)>0$, we have 
\begin{equation}\label{eqn:estbisss}
c\|\mathbf B\|_X\le\kappa_X+\mu n^{1/p}\|\mathcal B\|_X,
\end{equation}
with 
$$\kappa_X=\kappa_X(W):=\|W\|_X\|1/W\|_\infty$$ 
and $\mu$ as above. 
\end{thm} 

\begin{proof} Of course, it suffices to prove \eqref{eqn:estgar}. This done, \eqref{eqn:estbisss} 
will follow readily, the norms $\mathcal N(\cdot)$ and $\|\cdot\|_X$ being equivalent. 
\par As in \cite{DLoc}, we begin by verifying that the ratio $W\mathcal B^n/\mathbf B$ 
is analytic on $\D$ (and in fact on $\D\cup\T$). We need not worry 
about the zeros of $\mathbf B$ whose multiplicity is at most $n$, 
since these are obviously killed by the numerator, $W\mathcal B^n$. 
So let $z_0\in\D$ be a zero of multiplicity $k$, $k>n$, for $\mathbf B$. 
Then there is a $j\in\{0,\dots,n+1\}$ such that $B_j$ vanishes to order $k$ at $z_0$, 
and so does $f_j$. Expanding the determinant \eqref{eqn:wronskiannn} along the column which  
contains $f_j,\dots,f_j^{(n)}$, while noting that $f_j^{(l)}$ vanishes to order $k-l$ at $z_0$, 
we see that $W$ has a zero of multiplicity $\ge k-n$ at $z_0$. (In case $j=n+1$, 
one should observe that, by \eqref{eqn:summm}, the determinant remains unchanged upon 
replacing any one of its columns by $(f_{n+1},\dots,f^{(n)}_{n+1})^T$.) And since 
$\mathcal B$ has a zero at $z_0$, it follows that $W\mathcal B^n$ vanishes at least to 
order $k$ at that point. 
\par We conclude that $W\mathcal B^n$ is indeed divisible by $\mathbf B$. In other words, we have 
\begin{equation}\label{eqn:burundukkk}
W\mathcal B^n=F\mathbf B, 
\end{equation}
where $F$ is analytic on $\D$. In addition, this $F$ is analytic across $\T$, because the other 
factors in \eqref{eqn:burundukkk} have this property and because 
\begin{equation}\label{eqn:modulusone}
|\mathbf B|=|\mathcal B|=1\quad\text{\rm on }\T. 
\end{equation}
Factoring $F$ canonically (see \cite[Chapter II]{G}), we 
write $F=\mathcal I\mathcal O$, where $\mathcal I$ is inner and $\mathcal O$ is outer. 
Furthermore, a glance at \eqref{eqn:burundukkk} and \eqref{eqn:modulusone} reveals that $|F|=|W|$ 
on $\T$, so the outer factor $\mathcal O=\mathcal O_{|F|}$ coincides with $\mathcal O_{|W|}$. An 
application of \eqref{eqn:divproppp} with $h=\mathcal O_{|W|}\mathbf B$ and $\th=\mathcal I$ now 
shows that 
$$\mathcal N(W\mathcal B^n)
=\mathcal N(F\mathbf B)
=\mathcal N\left(\mathcal O_{|W|}\mathcal I\mathbf B\right)
\ge\mathcal N\left(\mathcal O_{|W|}\mathbf B\right).$$
This and \eqref{eqn:asympforbis} together imply 
\begin{equation}\label{eqn:wrobelowww}
\begin{aligned}
\mathcal N(W\mathcal B^n)
&\ge\mathcal N\left(\mathcal O_{|W|}\mathbf B\right)\\
&\ge\sup_{z\in\D}\f{\left|\mathcal O_{|W|}(z)\right|\cdot\{1-|\mathbf B(z)|^p\}^{1/p}}{k(z)}\\
&\ge\left(\inf_{z\in\D}\left|\mathcal O_{|W|}(z)\right|\right)\cdot
\sup_{z\in\D}\f{\{1-|\mathbf B(z)|^p\}^{1/p}}{k(z)}\\
&=\left(\inf_{z\in\D}\left|\mathcal O_{|W|}(z)\right|\right)\cdot\mathcal N(\mathbf B).
\end{aligned}
\end{equation}
We further observe that 
$$1/\mathcal O_{|W|}=\mathcal O_{1/|W|}\in H^\infty$$ 
(because $1/W\in L^\infty(\T)$) and 
$$\sup_{z\in\D}\left|\mathcal O_{|W|}(z)\right|^{-1}
=\left\|1/\mathcal O_{|W|}\right\|_\infty=\|1/W\|_\infty,$$ 
whence 
$$\inf_{z\in\D}\left|\mathcal O_{|W|}(z)\right|=\|1/W\|^{-1}_\infty.$$
Substituting this into \eqref{eqn:wrobelowww}, we obtain 
\begin{equation}\label{eqn:wrobelll}
\mathcal N(W\mathcal B^n)\ge\|1/W\|^{-1}_\infty\cdot\mathcal N(\mathbf B).
\end{equation}

\par Another application of \eqref{eqn:asympforbis} (with $c_p=1$), coupled with the elementary 
inequality 
$$1-t^n\le n(1-t)\quad\text{\rm for}\quad t\in[0,1],$$ 
yields 

\begin{equation*}
\begin{aligned} 
\mathcal N(W\mathcal B^n)&\le\mathcal N(W)+
\sup_{z\in\D}\f{|W(z)|\cdot\{1-|\mathcal B(z)|^{np}\}^{1/p}}{k(z)}\\
&\le\mathcal N(W)+\|W\|_\infty\cdot n^{1/p}\cdot\sup_{z\in\D}\f{\{1-|\mathcal B(z)|^p\}^{1/p}}{k(z)}\\
&=\mathcal N(W)+\|W\|_\infty\cdot n^{1/p}\cdot\mathcal N(\mathcal B).
\end{aligned}
\end{equation*} 
Thus, 
\begin{equation}\label{eqn:wroaboveee}
\mathcal N(W\mathcal B^n)\le\mathcal N(W)+n^{1/p}\|W\|_\infty\mathcal N(\mathcal B).
\end{equation}

\par Finally, we combine \eqref{eqn:wrobelll} and \eqref{eqn:wroaboveee} to get  
$$\mathcal N(\mathbf B)\le\|1/W\|_\infty\cdot\left\{\mathcal N(W)
+n^{1/p}\|W\|_\infty\mathcal N(\mathcal B)\right\},$$ 
which is the required estimate \eqref{eqn:estgar}. 
\end{proof}


\medskip

\begin{thebibliography}{12}

\bibitem{BM} W. D. Brownawell and D. W. Masser, {\it Vanishing sums in function fields}, 
Math. Proc. Cambridge Philos. Soc. \textbf{100} (1986), 427--434. 

\bibitem{DActa} K. M. Dyakonov, {\it Equivalent norms on Lipschitz-type spaces 
of holomorphic functions}, Acta Math. \textbf{178} (1997), 143--167. 

\bibitem{DAdv} K. M. Dyakonov, {\it Holomorphic functions and quasiconformal 
mappings with smooth moduli}, Adv. Math. \textbf{187} (2004), 146--172. 

\bibitem{DCRM} K. M. Dyakonov, {\it An $abc$ theorem on the disk}, C. R. Math. 
Acad. Sci. Paris \textbf{348} (2010), 1259--1261. 

\bibitem{DLoc} K. M. Dyakonov, {\it Zeros of analytic functions, with or without 
multiplicities}, Math. Ann. \textbf{352} (2012), 625--641. 

\bibitem{G} J. B. Garnett, {\it Bounded analytic functions, Revised first edition}, 
Springer, New York, 2007. 

\bibitem{GT} A. Granville and T. J. Tucker, {\it It's as easy as $abc$}, 
Notices Amer. Math. Soc. \textbf{49} (2002), 1224--1231. 

\bibitem{GH} G. G. Gundersen and W. K. Hayman, {\it The strength of Cartan's version 
of Nevanlinna theory}, Bull. London Math. Soc. \textbf{36} (2004), 433--454. 

\bibitem{K} P. Koosis, {\it Introduction to $H_p$ spaces, Second edition}, Cambridge 
University Press, Cambridge, 1998. 

\bibitem{L} S. Lang, {\it Old and new conjectured Diophantine inequalities}, 
Bull. Amer. Math. Soc. (N.S.) \textbf{23} (1990), 37--75. 

\bibitem{M} R. C. Mason, {\it Diophantine equations over function fields}, London 
Math. Soc. Lecture Note Series 96, Cambridge Univ. Press, 1984. 

\bibitem{ShSm} T. Sheil-Small, {\it Complex polynomials}, Cambridge Studies in Advanced 
Mathematics, 75, Cambridge University Press, Cambridge, 2002. 

\bibitem{St} W. W. Stothers, {\it Polynomial identities and Hauptmoduln}, Quart. J. Math. 
Oxford Ser. (2) \textbf{32} (1981), 349--370. 

\end{thebibliography}
\end{document}